\documentclass{amsart}
\usepackage{verbatim}
\newcommand{\cal}{\mathcal}

\newcommand{\rd}{{\mathbb R}^d}

\newcommand{\R}{{\mathbb R}}
\newcommand{\N}{{\mathbb N}}

\newcommand{\Ha}{{\cal H}}
\newcommand{\bd}{\partial}

\newcommand{\eps}{\varepsilon}
\newcommand{\sL}{\cal{L}}
\newcommand{\sS}{\cal{S}}
\newcommand{\sM}{\cal{M}}
\newcommand{\usS}{\overline{\cal S}}
\newcommand{\usM}{\overline{\cal M}}
\newcommand{\lsS}{\underline{\cal S}}
\newcommand{\lsM}{\underline{\cal M}}
\newcommand{\ldim}{\underline{\dim}}
\newcommand{\udim}{\overline{\dim}}
\newcommand{\conv}{{\rm conv}}

\newtheorem{Theorem}{Theorem}
\newtheorem{Proposition}[Theorem]{Proposition}

\newtheorem{question}[Theorem]{Question}
\newtheorem{Remark}[Theorem]{Remark}

\numberwithin{equation}{section}
\numberwithin{Theorem}{section}

\begin{document}
\title[Characterization of Minkowski measurability]{Characterization of Minkowski measurability in terms of surface area}
\author{Jan Rataj}
\address{Charles University, Faculty of Mathematics and Physics, Sokolovsk\'a 83, 18675 Praha 8, Czech Republic}
\author{Steffen Winter}
\address{Karlsruhe Institute of Technology, Department of Mathematics, 76128 Karlsruhe, Germany}
\thanks{Acknowledgement: The authors were supported by a cooperation grant of the Czech and the German science foundation, GACR project no.\ P201/10/J039 and DFG project no.\ WE 1613/2-1.}
\date{\today}
\begin{abstract}
The $r$-parallel set to a set $A$ in Euclidean space consists of all points with distance at most $r$ from $A$.
Recently, the asymptotic behaviour of volume and surface area of the parallel sets as $r$ tends to $0$ has been studied
and some general results regarding their relations have been established. Here we complete the picture regarding the resulting notions of Minkowski content and S-content.
In particular, we show that a set is Minkowski measurable if and only if it is S-measurable, i.e.\ if and only if its S-content is positive and finite, and that positivity and finiteness of the
lower and upper Minkowski contents imply the same for the S-contents and vice versa. The results are formulated in the
more general setting of Kneser functions. Furthermore, the relations between Minkowski and S-contents are studied for more general gauge functions. The results are applied to simplify the proof of the Modified Weyl-Berry conjecture in dimension one.
\end{abstract}
\maketitle

\section{Introduction}

Let $A$ be a bounded subset of $\rd$ and $r>0$. Denote by $d_A$ the (Euclidean) distance function of the set $A$, and by
$$
A_r:=\{ z\in\rd:\, d_A(z)\leq r\}
$$
 its $r$-parallel set (or  $r$-parallel neighbourhood). For $t\ge 0$, denote by ${\cal H}^{t}$ the $t$-dimensional Hausdorff measure. Let $V_A(r)={\cal H}^d(A_r)$ be the volume of the $r$-parallel set.

Stach\'o \cite{Stacho} showed (using the results of Kneser \cite{Kneser}) that the derivative $(V_A)'(r)$ of $V_A(r)$ exists for all $r>0$ except countably many and that the left and right derivatives exist at any $r>0$ and satisfy $(V_A)'_-(r)\geq (V_A)'_+(r)$.  Combining some results of Stach\'o~\cite{Stacho} and  Hug, Last and Weil \cite{HLW04}, it was shown in \cite{rw09} using a rectifiability argument that, for all $r>0$ except countably many,  the relation
\begin{equation} \label{vol_deriv}
V_A'(r) ={\cal H}^{d-1}(\partial A_r)
\end{equation}
holds.

In \cite{rw09}, also the limiting behaviour of $V_A(r)$ and ${\cal H}^{d-1}(\partial A_r)$ as $r\to 0$ was studied for arbitrary bounded sets $A\subset \R^d$
 and some close relations were established between the resulting notions of Minkowski content and S-content.

Recall that the \emph{$s$-dimensional lower and upper Minkowski content} of a compact set $A\subset\R^d$ are defined by
\[
\underline{\cal M}^s(A):=\liminf_{r\to 0} \frac{V_A(r)}{\kappa_{d-s}r^{d-s}} \quad \text{ and } \quad
\overline{\cal M}^s(A):=\limsup_{r\to 0} \frac{V_A(r)}{\kappa_{d-s}r^{d-s}},
\]
where $\kappa_t:=\pi^{t/2}/\Gamma(1+\frac t2)$. (If $t$ is an integer, then $\kappa_t$ is the volume of a unit $t$-ball).
If $\underline{\cal M}^s(A)=\overline{\cal M}^s(A)$, then the common value ${\cal M}^s(A)$ is refered to as the \emph{$s$-dimensional Minkowski content} of $A$. If the Minkowski content ${\cal M}^s(A)$ exists and is positive and finite, then the set $A$ is called \emph{($s$-dimensional) Minkowski measurable}.
We denote by
\[
\underline{\dim}_M A:=\inf\{t\ge 0 : \underline{\cal M}^s(A)=0\}=\sup\{t\ge 0 :\underline{\cal M}^s(A)=\infty\}
\]
and
\[
\overline{\dim}_M A=\inf\{t\ge 0 :\overline{\cal M}^s(A)=0\}=\sup\{t\ge 0 :\overline{\cal M}^s(A)=\infty\}
\]
the \emph{lower} and \emph{upper Minkowski dimension} of $A$.
Minkowski measurability plays an important role for instance in connection with the Weyl-Berry conjecture, see Section~\ref{sec:one-dim}.

In analogy with the Minkowski content, the \emph{upper and lower S-content} (or \emph{surface area based content}) of $A$ was introduced in \cite{rw09}, for $0\leq s<d$, by
\[
\underline{\cal S}^s(A):=\liminf_{r\to 0} \frac{{\cal H}^{d-1}(\bd A_r)}{(d-s)\kappa_{d-s}r^{d-1-s}}
\quad \text{ and } \quad
\overline{\cal S}^s(A):=\limsup_{r\to 0} \frac{{\cal H}^{d-1}(\bd A_r)}{(d-s)\kappa_{d-s}r^{d-1-s}},
\]
respectively. If both numbers coincide, the
common value is denoted by ${\cal S}^s(A)$ and called ($s$-dimensional) \emph{S-content} of $A$. It is convenient to set ${\cal S}^d(A):=0$ for completeness, which is justified by the fact that $\lim_{r\to 0}r{\cal H}^{d-1}(\bd A_r)=0$ for arbitrary bounded sets $A\subset\R^d$, cf.~\cite{Kneser}.
If ${\cal S}^s(A)$ exists and is positive and finite, then we call the set $A$ ($s$-dimensional) \emph{S-measurable}.
The numbers
\[
\underline{\dim}_S A:=\sup\{0\le t\le d: \underline{\cal S}^t(A)=\infty\}=\inf\{0\le t\le d: \underline{\cal S}^t(A)=0\}
\]
and
\[
\overline{\dim}_S A:=\sup\{0\le t\le d: \overline{\cal S}^t(A)=\infty\}=\inf\{0\le t\le d: \overline{\cal S}^t(A)=0\}
\]
are the \emph{lower and upper surface area based dimension} or \emph{S-dimension} of $A$, respectively.  Obviously, $\underline{\dim}_S A\le \overline{\dim}_S A$, and if equality holds, the common value will be regarded as the \emph{surface area based dimension} (or \emph{S-dimension}) of $A$ and denoted by $\dim_S A$.

In view of equation \eqref{vol_deriv}, it is apparent that Minkowski contents and S-contents of a set $A$ must be closely related. In \cite{rw09}, some precise results regarding this relation have been obtained, which are summarized as follows.
\begin{Theorem}\label{main-rw09} \cite[Corollaries~3.2, 3.4, 3.6 and Proposition 3.7]{rw09}\\
Let $A\subset\R^d$ be bounded and $s\in[0,d]$. Then
\begin{align} \label{eq:ucont}
 \frac{d-s}{d} \usS^s(A)\leq \usM^s(A)\leq \usS^s(A),
\end{align}
where for $s=d$ the left inequality is trivial and the right inequality holds only in case $V_A(0)=0$. As a consequence, $$\udim_M A= \udim_S A\,,$$ whenever $V_A(0)=0$. Furthermore,
 \begin{align}\label{eq:lcont}
\lsS^s(A)\leq \lsM^s(A)\leq c_{d,s} \left[\lsS^{s\frac{d-1}{d}}(A)\right]^\frac{d}{d-1},
\end{align}
    where for the right hand inequality we assume $d>1$ and where the constant $c_{d,s}$ just depends on the dimensions $s$ and $d$. As a consequence,
    \begin{align}\label{eq:ldim}
    \ldim_S A\le \ldim_M A \le \frac{d}{d-1} \ldim_S A.
    \end{align}
\end{Theorem}
Note that there is a fundamental difference between upper and lower contents. While the upper contents differ at most by a positive constant implying in particular the equivalence of the upper dimensions, the lower Minkowski content is in general only bounded from above by an S-content of some different dimension. This allows different lower dimensions. It was shown in \cite{w10}, that there exist indeed sets for which lower Minkowski dimension and lower S-dimension are different. Moreover, the constants given in \eqref{eq:ldim} were shown to be optimal.

In this note, we complete the picture concerning the relations between Minkowski contents and S-contents. We show that the existence of the Minkowski content (as a positive and finite number) is equivalent to the existence of the corresponding S-content, and that both numbers coincide in this case. In particular, this allows to characterize Minkowski measurability in terms of S-measurability. Moreover, while the positivity of lower $s$-dimensional Minkowski contents is in general not enough to conclude the positivity of the corresponding S-content, we show that the assumption of both positivity and finiteness of the upper and lower $s$-dimensional Minkowski contents is sufficient for the corresponding (upper and lower) S-contents to be positive and finite as well. Hence two sided bounds for Minkowski contents imply two sided bounds for the S-contents and vice versa, see Section~\ref{sec:two-side}.

In Section~\ref{sec:general-gauge}, we study also contents with more general gauge functions. Motivated by an open question regarding the asymptotics of area and boundary length of the parallel sets of the Brownian path in $\R^2$, we study generalized Minkowski contents, where the dimensions $s$ (corresponding to the gauge functions $h(r)=r^{d-s}$) are replaced by more general gauge functions $h:(0,\infty)\to(0,\infty)$ (e.g.\ $h(r)=|\log r|^{-1}$ as in the case of the Brownian path in the plane, cf.~\cite{LeGall88}). Generalized Minkowski contents are known in the literature, see e.g.~\cite{LapHe1,LapHe2,tricot,Zu05}.
We introduce the corresponding generalization of the S-contents and extend the relations between Minkowski and S-contents to the generalized counterparts. In particular, we show for a large class of gauge functions that the existence of the generalized Minkowski content $\sM(h;A)$ is equivalent to the existence of the corresponding generalized S-content $\sS(h';A)$, where $h'$ is the derivative of $h$ provided $h$ is differentiable.
Although our results cover a large class of gauge functions, unfortunately, they do not cover the case of the Brownian path. Our methods which are based on the Kneser property do not apply in this particular case. In Section~\ref{sec:ex}, some examples of  Kneser functions are constructed which indicate that the expected relation between the generalized contents may actually fail. At least they show that the underlying result for Kneser functions is not valid in this case.

Finally, in Section~\ref{sec:one-dim}, we study subsets of $\R$ and demonstrate how the results in this note can be used to simplify some essential parts of the proof derived by Lapidus and Pomerance in \cite{LapPo93} of the Modified Weyl-Berry conjecture in dimension one.

\section{Two sided bounds} \label{sec:two-side}

In \cite{rw09}, we considered the relation between either the two upper contents or the two lower contents, i.e.\ we tried to establish one-sided bounds, in which we succeeded in the case of the upper contents, but which turned out to be impossible for the lower contents, see Theorem~\ref{main-rw09} above. While the lower S-content is always a lower bound for the lower Minkowski content, cf.~\eqref{eq:lcont}, it is impossible to bound the lower S-content from below by the lower Minkowski content, if nothing is known about the upper Minkowski content. Even the lower S-dimension can be strictly smaller than the lower Minkowski dimension.
Now we consider upper and lower contents together. Assuming both lower and upper Minkowski content (of the same dimension $D$) to be positive and finite, one can conclude the same for the S-contents. We will derive this from a statement on Kneser functions formulated in Proposition~\ref{lem:two-side} below. Recall that a function $f:(0,\infty)\to(0,\infty)$ is called \emph{Kneser function of order $d\ge 1$}, if for all $0<a\le b<\infty$ and $\lambda\geq 1$,
$$
f(\lambda b)- f(\lambda a)\leq \lambda^d \left( f(b)-f(a)\right)\,.
$$

 Stach\'o observed that for a Kneser function $f$ of order $d\ge 1$, the function $f'_+(r) r^{1-d}$ is monotone decreasing (cf.~\cite[Theorem 1]{Stacho}). It is not difficult to see that the same is true for the function $f'_-(r) r^{1-d}$. That is, for $0<t\le r$,
\begin{equation} \label{eqn:decreasing}
\frac{f'_-(t)}{t^{d-1}}\ge \frac{f'_-(r)}{r^{d-1}} \text{ and } \frac{f'_+(t)}{t^{d-1}}\ge \frac{f'_+(r)}{r^{d-1}}.
\end{equation}
Note that a certain converse to Proposition~\ref{lem:two-side} was formulated in \cite[Proposition~3.1]{rw09}, see also
Proposition~\ref{prop-rw09} below.
\begin{Proposition} \label{lem:two-side}
Let $f$ be a Kneser function of order $d\ge1$. If
\begin{equation}\label{eqn:two-side1}
0<\liminf_{r\to 0} \frac{f(r)}{r^s}\le\limsup_{r\to 0} \frac{f(r)}{r^s}<\infty
\end{equation}
for some $s\in(0,\infty)$, then
$$
0<\liminf_{r\to 0} \frac{f'_+(r)}{r^{s-1}}\le \limsup_{r\to 0} \frac{f'_-(r)}{r^{s-1}}<\infty.
$$
\end{Proposition}
\begin{proof}
It follows from the assumption that there exist $0<m<M<\infty$ and $r_0>0$ so that, for all $r\in(0,r_0)$,
\begin{align} \label{eqn:assumption}
m r^s \le f(r) \le M r^s.
\end{align}
Choose $a>1$ so that $m a^s-M>0$. By \eqref{eqn:decreasing}, we have for any $r>0$
$$
f(ar)-f(r)=\int_r^{ar} f'_+(t) dt \le \int_r^{ar} f'_+(r)\left(\frac tr\right)^{d-1} dt=f'_+(r) r \frac{a^{d}-1}{d}.
$$
Moreover, we conclude from \eqref{eqn:assumption} that, for $ar\le r_0$,
$$
f(ar)-f(r)\ge (m a^s-M) r^s.
$$
Combining the last two inequalities, we obtain
$$
\frac{f'_+(r)}{r^{s-1}}\ge \frac d{a^d-1} (m a^s -M)>0
$$
for $r\in(0,\frac{r_0}a)$, and thus $\liminf_{r\to 0} \frac{f'_+(r)}{r^{s-1}}>0$.

A similar, slightly simpler argument works for the finiteness of the corresponding limes superior. Choosing some $0<b<1$, we have on the one hand $f(r)-f(br)\le f(r)\le M r^s$, by \eqref{eqn:assumption}, and on the other hand
\begin{equation}\label{eqn:f_}
f(r)-f(br)=\int_{br}^{r} f'_-(t) dt \ge f'_-(r) r \frac{1-b^d}d\,,
\end{equation}
by \eqref{eqn:decreasing}. Combining both inequalities, we get $\frac{f'_-(r)}{r^{s-1}}\le \frac {Md}{1-b^d}$, which implies the finiteness of the limsup.

Note that the first part of the assumption (that the limes inferior for $f$ is positive) is not needed for this conclusion. Since $b$ can be chosen arbitrarily close to $0$ and $M$ arbitrarily close to $\limsup_{r\to 0} \frac{f(r)}{r^s}$, we obtain the relation
\begin{align}\label{eqn:bound}
\limsup_{r\to 0} \frac{f'_-(r)}{r^{s-1}}\le d \limsup_{r\to 0} \frac{f(r)}{r^s}.
\end{align}
\end{proof}
We point out that the statement on the limes superior can also be derived from \cite[Lemma 3.5]{rw09}. The latter is formulated for the volume function $V_A$ but extends to arbitrary Kneser functions (cf.\ also \cite[p.10, first Remark]{rw09}). However, the argument given here is simpler. The constant obtained in \eqref{eqn:bound} is the same.

\begin{Theorem} \label{thm:twoside}
Let $A$ be a bounded subset of $\R^d$ and $D\in[0,d)$. Then
\begin{align}\label{eqn:twoside3}
0<\lsM^D(A)\le\usM^D(A)<\infty
\end{align}
if and only if
\begin{align}\label{eqn:twoside4}
0<\lsS^D(A)\le\usS^D(A)<\infty.
\end{align}
In this case, one has in particular $\dim_M A=\dim_S A=D$.
\end{Theorem}
\begin{proof} Assume first that \eqref{eqn:twoside3} holds.
Let $f(r):=V_A(r)$ and recall that $f$ is a Kneser function of order $d$ for each bounded set $A$ in $\R^d$. Set $s:=d-D$. The assumptions imply that $s>0$ and that the hypothesis \eqref{eqn:two-side1} of Proposition~\ref{lem:two-side} is satisfied. Hence, applying this proposition, we obtain
$$
0<\liminf_{r\to 0} \frac{(V_A)'_+(r)}{r^{d-D-1}}\le \limsup_{r\to 0} \frac{(V_A)'_-(r)}{r^{d-D-1}}<\infty.
$$
Now note that $S_A(r)\ge (V_A)'_+(r)$ for each $r>0$, which implies the positivity of $\lsS^D(A)$, and that $S_A(r)\le(V_A)'_-(r)$ for each $r>0$, which allows to conclude the finiteness of $\usS^D(A)$. This proves the implication \eqref{eqn:twoside3} $\Rightarrow$ \eqref{eqn:twoside4}.

The reverse implication follows directly from Theorem~\ref{main-rw09}, more precisely from the second inequality in \eqref{eq:ucont} and the first inequality in \eqref{eq:lcont}.
\end{proof}

Refining the argument in the proof of Proposition~\ref{lem:two-side}, we will now establish, that the existence of the Minkowski content $\sM^D(A)$ of a bounded set $A\subset \R^d$ of Minkowski dimension $D\in[0,d)$ implies the existence of its S-content $\sS^D(A)$. This extends a result in \cite{rh10} (where this was shown for $D\le d-1$) and clarifies and simplifies its proof.
\begin{Proposition} \label{lem:lim-exists}
Let $f$ be a Kneser function of order $d\ge1$. If
\begin{equation}\label{eqn:lim-exists}
\lim_{r\to 0} \frac{f(r)}{r^s}=C
\end{equation}
for some constants $s,C\in(0,\infty)$, then
$$
\lim_{r\to 0} \frac{f'_+(r)}{s r^{s-1}}=\lim_{r\to 0} \frac{f'_-(r)}{s r^{s-1}}=C.
$$
\end{Proposition}
\begin{proof}
Let $\eps\in(0,\min\{1,C\})$. Choose $a>1$ so that
\begin{equation}\label{eqn:choice-of-a}
a^s \ge \frac{C+\sqrt{\eps}}{C-\sqrt{\eps}}.
\end{equation}
Since $\sqrt{\eps}>\eps$ and thus $C+\sqrt{\eps}>C+\eps$ and  $C-\sqrt{\eps}<C-\eps$, we have
$$
(C-\eps)a^s - (C+\eps)>0.
$$
Repeating the argument of the first part of the proof of Proposition~\ref{lem:two-side} with $m:=C-\eps$ and $M:=C+\eps$, we infer that there exists some $r_0=r_0(\eps)$ such that
$$
\frac{f'_+(r)}{r^{s-1}}\ge \frac d{a^d-1} ((C-\eps)a^s - (C+\eps))=dC \frac{a^s-1}{a^d-1}-d\eps\frac{a^s+1}{a^d-1}>0
$$
for each $r\in(0,r_0)$. Hence
$$
\liminf_{r\to 0}\frac{f'_+(r)}{r^{s-1}}\ge dC \frac{a^s-1}{a^d-1}-d\eps\frac{a^s+1}{a^d-1}
$$
for each $\eps\in(0,\min\{1,C\})$ and each $a>1$ satisfying \eqref{eqn:choice-of-a}. Now we choose $a$ such that equality holds in \eqref{eqn:choice-of-a} and let $\eps$ tend to 0. Then $a=a(\eps)$ converges to 1 and so does $a^t$ for each $t\in\R$.
Moreover, for the last term on the right hand side we have
$$
\limsup_{\eps\to 0}d\eps\frac{a^s+1}{a^d-1}\le \lim_{\eps\to 0} \frac{3d\eps}{\left(\frac{C+\sqrt{\eps}}{C-\sqrt{\eps}}\right)^{\frac ds}-1}=0\,,
$$
which is easily seen by applying L'H\^opital's rule. Using again L'H\^opital's rule, we conclude that
$$
\liminf_{r\to 0}\frac{f'_+(r)}{sr^{s-1}}\ge  \frac ds C \lim_{a\searrow 1} \frac{a^s-1}{a^d-1}=\frac ds C \lim_{a\searrow 1} \frac{s a^{s-1}}{da^{d-1}}=C\,.
$$
An analogous argument shows that
$$
\limsup_{r\to 0}\frac{f'_-(r)}{sr^{s-1}}\le C.
$$
(Here one can choose $b=b(\eps)$ such that $b^s=(C-\sqrt{\eps})/(C+\sqrt{\eps})$ and derive the inequality
$$
\limsup_{r\to 0} \frac{f'_-(r)}{r^{s-1}}\le dC \frac{1-b^s}{1-b^d}+ d\eps \frac{1+b^s}{1-b^d}
$$
from which the claim follows by letting again $\eps\to 0$.)

The assertion of the proposition is now obvious taking into account that $f'_+\le f'_-$.
\end{proof}
\begin{Theorem} \label{thm:Mmeas}
Let $A$ be a bounded subset of $\R^d$, $D\in[0,d)$ and $M\in(0,\infty)$. Then
$$
\sM^D(A)=M,
$$
if and only if
$$
\sS^D(A)=M.
$$
That is, the set $A$ is Minkowski measurable (of order $D$) if and only if it is S-measurable (of order $D$)
and in this case both ($D$-dimensional) contents coincide.
\end{Theorem}
\begin{proof}
Assume that  $\sM^D(A)=M$. Let $f(r):=V_A(r)$ and $s:=d-D$ and apply Proposition~\ref{lem:lim-exists}. Then $
\sS^D(A)=M
$ follows from the fact that
$(V_A)'_+\le S_A\le (V_A)'_-$. The reverse implication is a direct consequence of Theorem~\ref{main-rw09}.
\end{proof}


\section{General gauge functions}\label{sec:general-gauge}

In the definition of generalized Minkowski contents the renormalization quotient $r^{d-D}$ is replaced by a more general \emph{gauge function} $h$. This allows to characterize the convergence behaviour of the parallel volume on a much finer scale, which is particularly useful, when the ordinary Minkowski content (of the correct dimension) is zero or infinite. Generalized Minkowski contents have been studied for instance in~\cite{LapHe1,LapHe2,tricot,Zu05}.

For a continuous function $h:(0,\infty)\to(0,\infty)$, let
\[
\underline{\cal M}(h;A):=\liminf_{r\to 0} \frac{V_A(r)}{h(r)} \quad \text{ and } \quad
\overline{\cal M}(h;A):=\limsup_{r\to 0} \frac{V_A(r)}{h(r)},
\]
be the \emph{lower and upper generalized Minkowski content with respect to $h$}. Similarly, we define
the \emph{lower and upper generalized S-content with respect to $h$} by
\[
\underline{\cal S}(h;A):=\liminf_{r\to 0} \frac{S_A(r)}{h(r)} \quad \text{ and } \quad
\overline{\cal S}(h;A):=\limsup_{r\to 0} \frac{S_A(r)}{h(r)}.
\]
If the corresponding upper and lower limits coincide, their common value is denoted by ${\cal M}(h;A)$ and ${\cal S}(h;A)$, respectively.
In the sequel we will establish some relations between generalized contents. It turns out that again derivatives play an important role.
We formulate our results first for general Kneser functions (and their derivatives) and specialize them afterwards to relations between volume and boundary surface area (or Minkowski and S-content).

In \cite{rw09}, the following proposition was proved in order to establish bounds for the Minkowski content in terms of the S-content. The result allows to immediately extend these bounds to the generalized contents.

\begin{Proposition}\cite[Proposition~3.1 and Remark on p.10]{rw09} \label{prop-rw09}
Let $f$ be a Kneser function and let $h:(0,\infty)\to(0,\infty)$ be a differentiable function with $\lim_{r\to 0} h(r)=0$.
Assume that $h'$ is nonzero on some right neighbourhood of $0$.\\
Let $\underline{S}:=\liminf_{r\to 0} f'(r)/h'(r)$ and $\overline{S}:=\limsup_{r\to 0} f'(r)/h'(r)$. Then
$$
\underline{S}\le \liminf_{r\to 0}\frac{f(r)-f(0)}{h(r)}\le \limsup_{r\to 0}\frac{f(r)-f(0)}{h(r)}\le \overline{S}.
$$
In particular, if $\underline{S}=\overline{S}$, i.e.~if the limit $S:=\lim_{r\to 0}f'(r)/h'(r)\in[0,\infty]$ exists then
$\lim_{r\to 0}(f(r)-f(0))/h(r)$ exists as well and equals $S$.
\end{Proposition}

Proposition~\ref{prop-rw09} yields the following general relations between generalized Minkow\-ski contents and S-contents.
\begin{Theorem} \label{thm:1-gauge}
Let $A\subset\R^d$ be bounded with $V_A(0)=0$ and let $h:(0,\infty)\to(0,\infty)$ be a differentiable function with $\lim_{r\to 0} h(r)=0$.
Assume that $h'$ is nonzero on some right neighbourhood of $0$.
Then,
$$\underline{\cal S}(h';A)\le \underline{\cal M}(h;A)\le \overline{\cal M}(h;A)\le \overline{\cal S}(h';A).$$
\end{Theorem}
\begin{proof}
Apply Proposition~\ref{prop-rw09} to $f(r):=V_A(r)$.
\end{proof}

Now we establish bounds for generalized S-contents in terms of generalized Minkowski contents providing the natural counterpart to the above inequalities. Taking into account the results of the previous section, it seems reasonable to assume that both the upper and lower generalized Minkowski content are positive and bounded - at least for the derivation of the lower bound. Again, the inequalities are derived from a more general statement on Kneser functions. Unfortunately, the results below are restricted to gauge functions of the special form $h(r)=r^s g(r)$ with $s>0$ and some non-decreasing $g$. Although this covers most gauge functions appearing the literature, the interesting case $s=0$ (appearing e.g.\ for the path of Brownian motion in $\R^d$) is excluded.

\begin{Proposition} \label{lem:two-side-gauge}
Let $f$ be a Kneser function of order $d\ge1$. If
\begin{equation}\label{eqn:two-side1-gauge}
0<\liminf_{r\to 0} \frac{f(r)}{h(r)}\le\limsup_{r\to 0} \frac{f(r)}{h(r)}<\infty
\end{equation}
for some function $h:(0,\infty)\to(0,\infty)$ of the form $h(r)=r^s g(r)$ where $s\in(0,\infty)$ and $g$ is non-decreasing, then
$$
0<\liminf_{r\to 0} \frac{f'_+(r)}{r^{-1}h(r)}\le \limsup_{r\to 0} \frac{f'_-(r)}{r^{-1}h(r)}<\infty.
$$
Moreover, if additionally $g$ (and thus $h$) is differentiable and $\limsup_{r\to 0}\frac {r g'(r)}{g(r)}<\infty$, then
$$
0<\liminf_{r\to 0} \frac{f'_+(r)}{h'(r)}\le \limsup_{r\to 0} \frac{f'_-(r)}{h'(r)}<\infty.
$$
\end{Proposition}
\begin{proof}
By the assumption there exist $0<m<M<\infty$ and $r_0>0$ so that, for all $r\in(0,r_0)$,
\begin{align} \label{eqn:assumption-gauge}
m h(r) \le f(r) \le M h(r).
\end{align}
Choose $a>1$ so that $m a^s-M>0$ (as in the proof of Proposition~\ref{lem:two-side}). Then, by \eqref{eqn:assumption-gauge}, we get for $ar\le r_0$,
$$
f(ar)-f(r)\geq m (ar)^s g(ar)-M r^s g(r)\geq (m a^s-M) r^s g(r),
$$
since $g$ is non-decreasing.
Combining this with inequality \eqref{eqn:decreasing} (which still holds, since $f$ is a Kneser function), we obtain
$$
\frac{f'_+(r)}{r^{s-1}g(r)}\ge \frac d{a^d-1} (m a^s -M)>0
$$
for $r\in(0,\frac{r_0}a)$, and thus $\liminf_{r\to 0} \frac{f'_+(r)}{r^{-1}h(r)}>0$.
If additionally $g$ (and thus $h$) is differentiable, then $h'(r)=s r^{s-1}g(r)+r^s g'(r)=r^{s-1}g(r)(s+\frac{r g'(r)}{g(r)})$. Hence
$$
\liminf_{r\to 0}\frac{f'_+(r)}{h'(r)}=\liminf_{r\to 0}\frac{f'_+(r)}{r^{s-1}g(r)}\cdot\frac 1{s+\frac{r g'(r)}{g(r)}}\ge \liminf_{r\to 0}\frac{f'_+(r)}{r^{s-1}g(r)}\cdot \liminf_{r\to 0}\frac1{s+\frac{r g'(r)}{g(r)}},
$$
where the first $\liminf$ is positive as we have just shown and for the second one this follows from the assumption
$\limsup_{r\to 0}\frac {r g'(r)}{g(r)}<\infty$.

Again a similar but slightly simpler argument shows the finiteness of the corresponding limes superior. Choosing some $0<b<1$, we get from \eqref{eqn:assumption-gauge}
$$
f(r)-f(br)\le f(r)\le M h(r)=M r^s g(r).
$$
Combining this with inequality \eqref{eqn:f_} (which holds, since $f$ is a Kneser function), we obtain $\frac{f'_-(r)}{r^{s-1}g(r)}\le \frac {Md}{1-b^d}$, which implies the finiteness of the first $\limsup$. If again, $g$ (and thus $h$) is differentiable, we have
$$
\limsup_{r\to 0}\frac{f'_-(r)}{h'(r)}=\limsup_{r\to 0}\frac{f'_-(r)}{r^{s-1}g(r)}\cdot\frac 1{s+\frac{r g'(r)}{g(r)}}\leq \limsup_{r\to 0}\frac{f'_-(r)}{r^{s-1}g(r)}\cdot \limsup_{r\to 0}\frac 1{s+\frac{r g'(r)}{g(r)}}.
$$
The expression in the last $\limsup$ is trivially bounded from above by $1/s$ and thus we conclude the finiteness of $\limsup_{r\to 0}\frac{f'_-(r)}{h'(r)}$.
Note that the assumption on the limes inferior is not needed for the proof of the bounds for the $\limsup$'s and that also the last condition $\limsup_{r\to 0}\frac {r g'(r)}{g(r)}<\infty$ is not used in this part. Since $b$ can be chosen arbitrarily close to $0$ and $M$ arbitrarily close to $\limsup_{r\to 0} \frac{f(r)}{h(r)}$, we obtain the relation
\begin{align}\label{eqn:bound_gauge}
\limsup_{r\to 0} \frac{f'_-(r)}{r^{-1}h(r)}\le d \limsup_{r\to 0} \frac{f(r)}{h(r)}.
\end{align}
\end{proof}

\begin{Theorem} \label{cor:h-twoside}
Let $A$ be a bounded subset of $\R^d$ such that
$$
0<\lsM(h;A)\le\usM(h;A)<\infty,
$$
for some gauge function $h:(0,\infty)\to(0,\infty), r\mapsto r^{d-D} g(r)$ with $D\in[0,d)>0$ and $g$ non-decreasing.
Then
$$
0<\lsS(r^{-1}h; A)\le\usS(r^{-1}h;A)<\infty.
$$
If additionally, $g$ is differentiable and $\limsup_{r\to 0}\frac {r g'(r)}{g(r)}<\infty$, then
$$
0<\lsS(h'; A)\le\usS(h';A)<\infty.
$$
\end{Theorem}
\begin{proof}
Let $f(r):=V_A(r)$ and recall that $f$ is a Kneser function of order $d$ for each bounded set $A$ in $\R^d$. Set $s:=d-D$. The assumptions imply that $s>0$ and that the hypothesis \eqref{eqn:two-side1-gauge} of Proposition~\ref{lem:two-side-gauge} is satisfied. Hence, applying this proposition, we obtain
$$
0<\liminf_{r\to 0} \frac{(V_A)'_+(r)}{r^{d-D-1}g(r)}\le \limsup_{r\to 0} \frac{(V_A)'_-(r)}{r^{d-D-1}g(r)}<\infty.
$$
Now note that $S_A(r)\ge (V_A)'_+(r)$ for each $r>0$, which implies the positivity of $\lsS(r^{-1}h;A)$, and that $S_A(r)\le(V_A)'_-(r)$ for each $r>0$, which allows to conclude the finiteness of $\usS(r^{-1}h;A)$. If $g$ (and thus $h$) is differentiable and $\limsup_{r\to 0}\frac {r g'(r)}{g(r)}<\infty$, then the second assertion of Proposition~\ref{lem:two-side-gauge} can be applied. Again, the inequality $S_A(r)\ge (V_A)'_+(r)$, yields
the positivity of $\lsS(h';A)$ and the inequality  $S_A(r)\le(V_A)'_-(r)$ the finiteness of  $\usS(h';A)$.
\end{proof}

\begin{Remark}
\emph{
Note that for the finiteness of $\usS(r^{-1}h;A)$ (and $\usS(h';A)$, respectively), only the finiteness of $\usM(h;A)$ is required, while for the corresponding lower bounds the whole hypothesis is needed.
Moreover, as a corollary to the proof of Proposition~\ref{lem:two-side-gauge} we obtain the direct relation
$$
\usS(r^{-1}h; A)\le d \usM(h;A),
$$
and, in case $g$ is differentiable, also $\usS(h'; A)\le d \usM(h;A)
$.
}
\end{Remark}

Refining the argument in the proof of Proposition~\ref{lem:two-side-gauge}, we will now establish, that the existence of the generalized Minkowski content $\sM(h;A)$ of a bounded set $A\subset \R^d$ (of Minkowski dimension $D\in[0,d)$) implies the existence of its generalized S-content $\sS(h';A)$. This extends Proposition~\ref{lem:lim-exists} and its corollaries to the case of gauge functions.
\begin{Proposition} \label{lem:lim-exists-gauge}
Let $f$ be a Kneser function of order $d\ge1$. If
\begin{equation}\label{eqn:lim-exists-gauge}
\lim_{r\to 0} \frac{f(r)}{h(r)}=C
\end{equation}
for some function $h:(0,\infty)\to(0,\infty)$ of the form $h(r)=r^s g(r)$ where $s,C\in(0,\infty)$ are positive constants and $g$ is non-decreasing, then
$$
\lim_{r\to 0} \frac{f'_+(r)}{s r^{s-1}g(r)}=\lim_{r\to 0} \frac{f'_-(r)}{s r^{s-1}g(r)}=C.
$$
If $g$ is differentiable and $\lim_{r\to 0}\frac {r g'(r)}{g(r)}=0$, then
$$
\lim_{r\to 0} \frac{f'_+(r)}{h'(r)}=\lim_{r\to 0} \frac{f'_-(r)}{h'(r)}=C.
$$
\end{Proposition}
\begin{proof}
Let $\eps\in(0,\min\{1,C\})$. Choose $a>1$ so that
\begin{equation}\label{eqn:choice-of-a-gauge}
a^s \ge \frac{C+\sqrt{\eps}}{C-\sqrt{\eps}}.
\end{equation}
Since $\sqrt{\eps}>\eps$ and thus $C+\sqrt{\eps}>C+\eps$ and  $C-\sqrt{\eps}<C-\eps$, we have
$$
(C-\eps)a^s - (C+\eps)>0.
$$
Repeating the argument of the first part of the proof of Proposition~\ref{lem:two-side-gauge} with $m:=C-\eps$ and $M:=C+\eps$, we infer that there exists some $r_0=r_0(\eps)$ such that
$$
\frac{f'_+(r)}{r^{s-1}g(r)}\ge \frac d{a^d-1} ((C-\eps)a^s - (C+\eps))=dC \frac{a^s-1}{a^d-1}-d\eps\frac{a^s+1}{a^d-1}>0
$$
for each $r\in(0,r_0)$. Hence
$$
\liminf_{r\to 0}\frac{f'_+(r)}{r^{s-1}g(r)}\ge dC \frac{a^s-1}{a^d-1}-d\eps\frac{a^s+1}{a^d-1}
$$
for each $\eps\in(0,\min\{1,C\})$ and each $a>1$ satisfying \eqref{eqn:choice-of-a-gauge}. Now we choose $a$ such that equality holds in \eqref{eqn:choice-of-a} and let $\eps$ tend to 0. Then $a=a(\eps)$ converges to 1 and so does $a^t$ for each $t\in\R$.
Moreover, for the last term on the right hand side we have
$$
\limsup_{\eps\to 0}d\eps\frac{a^s+1}{a^d-1}\le \lim_{\eps\to 0} \frac{3d\eps}{\left(\frac{C+\sqrt{\eps}}{C-\sqrt{\eps}}\right)^{\frac ds}-1}=0\,,
$$
which is easily seen by applying L'H\^opital's rule. Using again L'H\^opital's rule, we conclude that
$$
\liminf_{r\to 0}\frac{f'_+(r)}{sr^{s-1}g(r)}\ge  \frac ds C \lim_{a\searrow 1} \frac{a^s-1}{a^d-1}=\frac ds C \lim_{a\searrow 1} \frac{s a^{s-1}}{da^{d-1}}=C\,.
$$
An analogous argument shows that
$$
\limsup_{r\to 0}\frac{f'_-(r)}{sr^{s-1}g(r)}\le C.
$$
(Here one can choose $b=b(\eps)$ such that $b^s=(C-\sqrt{\eps})/(C+\sqrt{\eps})$ and derive the inequality
$$
\limsup_{r\to 0} \frac{f'_-(r)}{r^{s-1}g(r)}\le dC \frac{1-b^s}{1-b^d}+ d\eps \frac{1+b^s}{1-b^d}
$$
from which the claim follows by letting again $\eps\to 0$.)

The first assertion of the proposition is now obvious taking into account that $f'_+\le f'_-$.
Now assume that $g$ (and thus $h$) is differentiable and that $\lim_{r\to 0}\frac {r g'(r)}{g(r)}=0$. Since $h'(r)=r^{s-1} g(r) (s+\frac{rg'(r)}{g(r)})$, we get on the one hand
$$
\liminf_{r\to 0}\frac{f'_+(r)}{h'(r)}=\liminf_{r\to 0}\frac{f'_+(r)}{s r^{s-1}g(r)}\cdot\frac 1{1+\frac 1s \frac{r g'(r)}{g(r)}}\ge C\cdot \liminf_{r\to 0}\frac1{1+\frac 1s \frac{r g'(r)}{g(r)}}=C,
$$
and on the other hand
$$
\limsup_{r\to 0}\frac{f'_-(r)}{h'(r)}=\limsup_{r\to 0}\frac{f'_-(r)}{s r^{s-1}g(r)}\cdot\frac 1{1+\frac 1 s\frac{r g'(r)}{ g(r)}}\leq C\cdot \limsup_{r\to 0}\frac 1{1+\frac 1s \frac{r g'(r)}{g(r)}}=C.
$$
\end{proof}

\begin{Theorem}
Let $A$ be a bounded subset of $\R^d$ and let $h:(0,\infty)\to(0,\infty)$ be a gauge function of the form $h(r)=r^{d-D}g(r)$ for some $D\in[0,d)$ and some non-decreasing, differentiable function $g$ satisfying $\lim_{r\to 0}\frac {r g'(r)}{g(r)}=0$. Let $M\in(0,\infty)$. Then
$$
\sM(h;A)=M,
$$
if and only if
$$
\sS(h';A)=M.
$$
That is, the set $A$ is generalized Minkowski measurable (with respect to $h$) if and only if it is generalized S-measurable (with respect to $h'$)
and in this case both contents coincide.
\end{Theorem}
\begin{proof}
Assume that  $\sM(h;A)=M$. Let $f(r):=V_A(r)$ and $s:=d-D$ and apply Proposition~\ref{lem:lim-exists-gauge} (for which the assumption $\lim_{r\to 0}\frac {r g'(r)}{g(r)}=0$ used). Then $
\sS(h';A)=M
$ follows from the fact that
$(V_A)'_+\le S_A\le (V_A)'_-$. The reverse implication is a direct consequence of Theorem~\ref{thm:1-gauge}. The special form of $h$ and the fact that $g$ is nondecreasing imply $\lim_{r\to 0} h(r)=0$ and $g'(r)\geq 0$. The latter implies also that $h'(r)=s r^{s-1} g(r)+ r^s g'(r)$ is non-zero for $r>0$. Hence the hypotheses of Theorem~\ref{thm:1-gauge} are satisfied.
\end{proof}

\begin{Remark}
\emph{
The condition $\lim_{r\to 0}\frac {r g'(r)}{g(r)}=0$ in the statement above is a very reasonable assumption. Roughly it means that the function $g$ grows slower than any power $r^t$, $t>0$ (since $(r^t)'=t r^{t-1}$ and so $\frac{r\cdot (r^t)'}{r^t}=t$). Thus a violation of this assumption essentially means that the exponent $s$ in $h(r)=r^s g(r)$ is not chosen correctly.
The condition
$\limsup_{r\to 0}\frac {r g'(r)}{g(r)}<\infty$ appearing in the statements before is even weaker.
Note that if the function $g$ is e.g.\ concave, then we have $\frac {r g'(r)}{g(r)}\le 1$ for each $r>0$ and so the latter condition is automatically satisfied. 
\\
Assuming differentiability of $h$ is also not really a restriction, since we are only interested in the asymptotics as $r\to 0$. For any continuous, non-decreasing function $h:(0,\infty)\to(0,\infty)$ there is a differentiable function $\tilde h:(0,\infty)\to(0,\infty)$ such that $\lim_{r\to 0} \frac {\tilde h(r)}{h(r)}=1$.
}
\end{Remark}

\section{Counterexamples} \label{sec:ex}

In \cite{rh10}, an example of a Kneser function $f$ was presented with the property $$
\lim_{r\to 0} \frac{f(r)}{h(r)}=1 \quad (\text{shortly, } f(r)\sim h(r) \text{ as } r\to 0_+)
$$
for the gauge function $h(x)=|\log x|^{-1}$, while at the same time the limit\\ $\lim_{r\to 0_+}f'(r)/h'(r)$ did not exist. We recall this construction here and modify it in order to get a Kneser function $f$ for which even $\liminf_{r\to 0_+}f'(r)/h'(r)=0$ and $\limsup_{r\to 0_+}f'(r)/h'(r)=\infty$ holds. The examples show in particular, that the assumptions on the gauge function $h$ in
Propositions~\ref{lem:two-side-gauge} and \ref{lem:lim-exists-gauge} cannot be dropped in general.

Consider the following scheme producing Kneser functions of order $2$. Let $r_i\searrow 0$ and $a_i\nearrow \infty$ be two monotone positive sequences such that
\begin{equation}  \label{*}
\sum_ia_ir_i^2<\infty.
\end{equation}
Let $f$ be such that
$$f'(x)=2a_ix,\quad x\in (r_{i+1},r_i),\quad i=1,2,\ldots .$$
Condition \eqref{*} guarantees that such an $f$ exists, namely,
\begin{eqnarray*}
f(x)&=&\sum_{j\geq i}\int_{r_{j+1}}^{r_j}2a_jy\, dy +\int_{r_i}^x 2a_iy\, dy\\
&=&\sum_{j\geq i}a_j(r_j^2-r_{j+1}^2) +\int_{r_i}^x 2a_iy\, dy,\quad x\in (r_{i-1},r_i).
\end{eqnarray*}
The monotonicity of $(a_i)$ ensures that $f$ is a Kneser function of order $2$.

In \cite{rh10}, the authors considered (up to some constant) the following case:
$$r_i=2^{-i},\quad a_i=\frac{2^{2i}}{i(i+1)},\quad i=1,2,\ldots .$$
It is not difficult to verify that
$$f(r_i)=\sum_{j\geq i}a_j(r_j^2-r_{j+1}^2)=\frac 34\sum_{j\geq i}\frac 1{j(j+1)}=\frac 3{4i} =\frac{3\log 2}4\frac 1{|\log r_i|}$$
and, choosing $h(r)=\frac{3\log 2}4\frac 1{|\log r|}$, from the monotonicity of $f$ we get
$$f(r)\sim h(r),\quad \text{ as } r\to 0+.$$
On the other hand, we have
\begin{eqnarray*}
f'(r_i-)&=&2a_ir_i=2\frac{2^i}{i(i+1)}=2\frac i{i+1}\frac {(\log2)^2}{r_i(\log r_i)^2} =\frac{8\log 2}{3}\frac{i}{i+1}h'(r_i),\\
f'(r_i+)&=&2a_{i-1}r_i=2\frac{2^i}{i(i+1)}=\frac 12\frac i{i+1}\frac {(\log 2)^2}{r_i(\log r_i)^2} =\frac{2\log 2}{3}\frac{i}{i+1}h'(r_i).
\end{eqnarray*}
Consequently,
$$\liminf_{r\to 0_+}\frac{f'(r)}{h'(r)}\leq \frac{2\log 2}3<\frac{8\log 2}3\leq\limsup_{r\to 0_+}\frac{f'(r)}{h'(r)}.$$
This shows that in Proposition~\ref{lem:lim-exists-gauge} the assumption $h(r)=r^sg(r)$ with some $s>0$ and $g$ non-decreasing cannot be omitted .

We shall now consider a modified version of the above example showing that also in Proposition~\ref{lem:two-side-gauge} the assumption on the gauge function cannot be relaxed. Consider the sequences
$$r_i=2^{-2^i},\quad a_i=\frac {2^{2^{i+1}}}{i(i+1)}$$
which are again monotone. The associated function $f$ is a Kneser function of order $2$. Again, condition \eqref{*} is clearly fulfilled. We have
\begin{eqnarray*}
f(r_i)&=&\sum_{j\geq i}a_j(r_j^2-r_{j+1}^2)=\sum_{j\geq i}\frac 1{j(j+1)}\left( 1-2^{-2^{j+1}}\right)\\
&\sim&\sum_{j\geq i}\frac 1{j(j+1)}=\frac 1i \quad \text{ as } i\to\infty.
\end{eqnarray*}
Thus, with the gauge function $h(r)=\log 2/\log|\log r|$, since
$$h(r_i)=\frac{\log 2}{i\log 2+\log\log 2}\sim \frac 1i,$$
we get $f(r_i)\sim h(r_i)$ as $i\to \infty$, and, using the monotonicity of $f$, it is not difficult to see that
$$f(r)\sim h(r),\quad r\to 0_+.$$
On the other hand, we have
$$h'(r_i)=\frac 1{\log^22}\frac{2^{2^i-i}}{i^2},$$
whereas
\begin{eqnarray*}
f'(r_i-)&=&2a_ir_i=2\frac{2^{2^i}}{i(i+1)},\\
f'(r_i+)&=&2a_{i-1}r_i=2\frac{1}{i(i+1)}.
\end{eqnarray*}
Thus,
$$0=\liminf_{r\to 0_+}\frac{f'(r)}{h'(r)}<\limsup_{r\to 0_+}\frac{f'(r)}{h'(r)}=\infty.$$

The counterexamples presented above are examples of Kneser functions. It is not clear whether there exist sets, having these functions as their volume function. It would be even more interesting, to obtain analogous examples of volume functions of sets. This seems to be a more difficult problem and we formulate it here as an open question.

\begin{question}
Does there exist a bounded set $A\subset\R^n$ and a gauge function $h$ (assumed to be differentiable without loss of generality) such that $A$ is generalized Minkowski measurable with respect to $h$, but not generalized S-measurable with respect to $h'$?
\end{question}

\section{Subsets of the line} \label{sec:one-dim}

We draw our attention to the case $d=1$ and relate the above results to the results of Lapidus and Pomerance in \cite{LapPo93}, where a characterization of Minkowski measurability is given in connection with the proof of the Modified Weyl-Berry conjecture in dimension one. In view of the previous sections, it is natural to add the criterion of S-measurability to the equivalent characterizations of Minkowski measurability given. It turns out that, using the S-content as an intermediate step and applying the above results, some of the proofs in \cite{LapPo93} can be significantly simplified. Moreover, the results below indicate that from a certain point of view, the  surface area of the parallel sets might be the proper object in higher dimensions to replace the fractal string associated to sets on the real line.

Recall that to any compact subset $F\subset\R$, one can associate a unique \emph{fractal string} $\sL=(l_j)_{j=1}^\infty$, that is, a nonincreasing sequence of real numbers $l_j$ encoding the lengths of the bounded complementary intervals of $F$.
When comparing the formulas below with the ones in \cite{LapPo93}, it should be kept in mind that here we have an additional normalization constant $\kappa_{d-s}$ in the definition of the $s$-dimensional Minkowski content.
In analogy with the notation in the previous section, for two positive sequences $(a_j)_{j\in\N}$ and $(b_j)_{j\in\N}$ we write $a_j\sim b_j$ as $j\to\infty$, if $\lim_{j\to\infty}a_j/b_j= 1$, and similarly $a_j\approx b_j$ as $j\to\infty$, if there are constants $c, C$ and $j_0$ such that $c\le a_j/b_j\le C$ for all $j\geq j_0$.

\begin{Theorem}\label{thm:one-dim}
Let $F\subset\R$ be a compact set with Minkowski dimension $\dim_M F =D\in(0,1)$ (i.e.~in particular with $\lambda(F)=0$) and let $\sL=(l_j)_{j=1}^\infty$ be the fractal string associated with $F$.\\
\emph{(a) Two sided bounds.} The following assertions are equivalent:
\begin{enumerate}
\item[(i)] $0<\lsM^D(F)\le\usM^D(F)<\infty$
\item[(ii)] $0<\lsS^D(F)\le\usS^D(F)<\infty$
\item[(iii)] $l_j\approx j^{-1/D}$ as $j\to\infty$
\end{enumerate}
\emph{(b) Criterion of Minkowski measurability.} The following assertions are equivalent:
\begin{enumerate}
\item[(i)] $F$ is Minkowski measurable,
\item[(ii)] $F$ is S-measurable, i.e., $0<\sS^D(F)<\infty$,
\item[(iii)] $l_j\sim L j^{-1/D}$ as $j\to\infty$ for some $L>0$.
\end{enumerate}
Under these latter assertions, Minkowski and S-content of $F$ are given by
\begin{align} \label{eqn:Scont-exists}
\sM^D(F)=\sS^D(F)=\frac{2^{1-D}}{\kappa_{1-D}}\frac{L^D}{1-D}.
\end{align}
\end{Theorem}
\begin{proof}
(a) The equivalence $(i)\Leftrightarrow(ii)$ is the case $d=1$ of Theorem~\ref{thm:twoside}. The equivalence $(i)\Leftrightarrow(iii)$ is part of Theorem~2.4 in \cite{LapPo93}. However, we give a simpler direct proof of the equivalence $(ii)\Leftrightarrow(iii)$, connecting the S-content directly to the asymptotics of the lengths $l_j$ in the associated fractal string $\sL$.

For a proof of the implication $(iii) \Rightarrow(ii)$, recall that $(iii)$ means there exist constants $c_1,c_2, j_0$ such that $c_1 < j \cdot l_j^D < c_2$ for all $j\ge j_0$. For easier comparison, we follow the notation in \cite{LapPo93} and let
\begin{align} \label{eqn:def-alpha-beta}
\alpha:=\liminf_{j\to\infty} l_j j^{1/D} \quad \text{ and } \quad \beta:=\limsup_{j\to\infty} l_j j^{1/D},
\end{align}
cf.~\cite[(3.5), p.48]{LapPo93}. Obviously, $0<\alpha\le\beta<\infty$. Furthermore, let $J(\eps):=\max\{j\in\N: l_j\ge\eps\}$, $\eps>0$.
For $2r\in(l_{j+1},l_j]$, we have $J(2r)=j$ and thus
\begin{align}\label{eqn:H0-J}
\Ha^0(\bd F_r)=2+ 2 J(2r)= 2+2j.
\end{align}
Hence
$$
2^{1-D} l_{j+1}^D (j+1)\le r^D \Ha^0(\bd F_r)\le 2^{1-D} l_j^D (j+1).
$$
We conclude that on the one hand
$$
\kappa_{1-D}\lsS^D(F)=\liminf_{r\searrow 0} \frac {\Ha^0(\bd F_r)}{(1-D) r^{-D}}\ge \frac{2^{1-D}}{1-D} \liminf_{j\to\infty} l_{j+1}^D (j+1)=\frac{2^{1-D}}{1-D} \alpha^D>0
$$
and, on the other hand
$$
\kappa_{1-D}\usS^D(F)=\limsup_{r\searrow 0} \frac {\Ha^0(\bd F_r)}{(1-D) r^{-D}}\le \frac{2^{1-D}}{1-D} \liminf_{j\to\infty} l_{j}^D j\cdot\frac {j+1}j=\frac{2^{1-D}}{1-D} \beta^D<\infty,
$$
cf.\ also \cite[cf.\ Proof of Theorem~3.1, equation (3.9)]{LapPo93}).

For a proof of the reverse implication $(ii)\Rightarrow(iii)$, we essentially employ the argument in \cite[Lemma~3.6]{LapPo93}, which gives the lower bound, and observe that a similar argument works for the upper bound: Letting $\alpha_j:=l_j j^{1/D}$, it is not difficult to see that $l_j=l_{j+1}$ implies $\alpha_j<\alpha_{j+1}$. This yields
\begin{align}\label{eqn:alpha-beta}
\alpha=\liminf_{j:l_j>l_{j+1}} \alpha_{j+1} \quad \text{ and, similarly,} \quad \beta=\limsup_{j:l_j>l_{j+1}} \alpha_j.
\end{align}
Suppose $(ii)$ holds. Then, in view of the relation $\Ha^0(\bd F_r)=2+ 2 J(2r)$, 
there exist positive constants $r_0,m,M$ such that $m\leq r^D J(r)\leq M$ for all $0<r\leq r_0$.
For $l_{j+1}<l_j$ and $r\in(l_{j+1}, l_{j}]$ one has $J(r)=j$ and thus $m\leq r^D j\leq M$, which implies
$$
m^{1/D} j^{-1/D}\leq r \leq M^{1/D} j^{-1/D}.
$$
Letting $r\to l_{j+1}$ in this equation, we get on the one hand $l_{j+1}\geq m^{1/D} j^{-1/D}>m^{1/D} (j+1)^{-1/D}$, which implies $\alpha_{j+1}>m^{1/D}$ for each $j$ with $l_j>l_{j+1}$. Thus
$$
\alpha=\liminf_{j:l_j>l_{j+1}} \alpha_{j+1} \ge m^{1/D}>0.
$$
On the other hand, we get for $r=l_j$, $l_j\leq M^{1/D} j^{-1/D}$ and so $\alpha_j\leq M^{1/D}$ for each $j$ with $l_j>l_{j+1}$. Hence
$$
\beta=\limsup_{j:l_j>l_{j+1}} \alpha_j\leq M^{1/D}.
$$

(b) The equivalence $(i)\Leftrightarrow(ii)$ is the case $d=1$ of Theorem~\ref{thm:Mmeas}. The equivalence $(i)\Leftrightarrow(iii)$ is proved in \cite[Theorems 3.1(b) and 4.1]{LapPo93}.  In particular, the latter Theorem has a long and technical proof. We prove the equivalence $(ii)\Leftrightarrow(iii)$ instead: The implication $(iii)\Rightarrow(ii)$ follows immediately from the proof of the same implication in (a) by setting $\alpha=\beta=L$.

For a proof of the reverse implication, we refine the argument of the corresponding proof in part (a). Assume
\begin{align} \label{eqn:Scont-exists}
\sS^D(F)=\frac{2^{1-D}}{\kappa_{1-D}}\frac{L^D}{1-D}
\end{align}
for some $L>0$, which implies that for each $\eps>0$ there exists $r_0>0$ such that
$$
L^D-\eps\leq 2^{D-1} r^D \Ha^0(\bd F_r)\leq L^D+\eps
$$
for $0<r\leq r_0$. It suffices to show $\alpha\geq L$ and $\beta\leq L$, where $\alpha$ and $\beta$ are as in \eqref{eqn:def-alpha-beta}.
Recalling that $\Ha^0(\bd F_r)=2+ 2 J(2r)$ and substituting $t=2r$, we infer
$$
L^D-\eps\le t^D (1+J(t))\le L^D+\eps
$$
for all $0<t\leq 2r_0$. Now, if $l_{j+1}<l_j$ and $t\in(l_{j+1},j_j]$, then $J(t)=j$ and so
\begin{align}\label{eqn:Scont-exists2}
(L^D-\eps)^{1/D}\leq t (1+j)^{1/D} \leq (L^D+\eps)^{1/D}.
\end{align}
Setting $t=l_j$ and taking into account \eqref{eqn:alpha-beta}, we conclude
$$
\limsup_{j\to\infty} l_j j^{1/D}\leq \limsup_{j:l_j>l_{j+1}} l_j (j+1)^{1/D}\leq (L^D+\eps)^{1/D}
$$
for each $\eps>0$ and thus, by letting $\eps\to 0$, $\beta\leq L$.
Similarly, by letting $t\to l_{j+1}$ in \eqref{eqn:Scont-exists2} and using again \ref{eqn:alpha-beta}, we get
$$
\liminf_{j\to\infty} l_j j^{1/D}=\liminf_{j:l_j>l_{j+1}} l_{j+1} (j+1)^{1/D}\geq (L^D-\eps)^{1/D}
$$
and so, by letting $\eps\to 0$, $\alpha>L$. This completes the proof of the implication $(ii)\Rightarrow(iii)$
in part (b) and thus of Theorem~\ref{thm:one-dim}.
\end{proof}

\begin{Remark}
\emph{
({\it Connection to the Modified Weyl-Berry (MWB) conjecture})
Let $\Omega=F^c\cap\conv(F)$ be the bounded open set consisting of the bounded complementary intervals of the compact set $F$. In \cite{LapPo93}, the Minkowski content is  connected to the following function $\delta:(0,\infty)\to (0,\infty)$, which describes the error if one tries to pack intervals of a fixed small length $l=1/x$ into the complementary intervals of $F$ (whose lengths are given by the associated fractal string $\sL=(l_j)_{j=1}^\infty$):
\begin{align}
\delta(x)=\sum_{j=1}^\infty l_j x -\sum_{j=1}^\infty [l_j x]=\sum_{j=1}^\infty \{l_j x\}.
\end{align}
Here $[z]$ and \{z\} denote the integer part and the fractional part of a number $z\in\R$, respectively.\\
According to \cite[Theorem 2.4]{LapPo93}, the following assertion is equivalent to the items $(i)-(iii)$ in part (a) of Theorem~\ref{thm:one-dim}:
$$
\delta(x)\approx x^D \text{ as } x\to\infty.
$$
Similarly, by \cite[Theorem~4.2]{LapPo93}, the assertions in part (b) of Theorem~\ref{thm:one-dim} imply that
\begin{align}\label{eqn:delta}
\delta(x)&\sim - \zeta(D)L^Dx^D, \text{ as } x\to\infty,
\end{align}
where $\zeta$ denotes the Riemann zeta-function. This relation is the  key to the proof of the MWB conjecture in dimension one in \cite{LapPo93},
which connects the geometry of the set $\Omega$ to its spectral properties (that is, to its sound).
Let $\lambda_1\leq \lambda_2\leq \ldots$ be the eigenvalues of the Dirichlet Laplacian $\Delta$ on $\Omega$ in increasing order and counted according to their multiplicities and let $N(\lambda):=\#\{k\geq 1: \lambda_i\le\lambda\}$ be the eigenvalue counting function of $\Delta$.
The MWB conjecture states that the second order asymptotic behavior of $N(\lambda)$ is governed by the Minkowski content of the boundary $F=\bd\Omega$ of $\Omega$. (The first order asymptotics is well known to be given by the so called \emph{Weyl term} $\varphi(\lambda)$ involving the volume of $\Omega$.) More precisely, for $d=1$, one has
$$
N(\lambda)=\varphi(\lambda)- c_{1,D}\sM^D(F)\lambda^{D/2} + o(\lambda^{D/2}) \text{ as } \lambda\to\infty,
$$
where $\varphi(\lambda)=\pi^{-1} V(\Omega)\lambda^{1/2}$ and $c_{1,D}=2^{D-1}\pi^{-D} \kappa_{1-D} (1-D)\zeta(D)$.
This is easily seen from \eqref{eqn:delta} and the relation
$$
\varphi(\lambda)-N(\lambda)=\sum_{j=1}^\infty l_jx -\sum_{j=1}^\infty [l_jx]=\delta(x)
$$
where $x=\sqrt{\lambda}/\pi$.
We refer to \cite{LapPo93} for more details on the resolution of the MWB conjecture in dimension one and to \cite{LapPo96} for its disproof in higher dimensions, see also \cite{LapvF06}.
Surprisingly, a certain converse of the implication above connecting \eqref{eqn:delta} to the assertions in Theorem~\ref{thm:one-dim}(b) is not true in the case $D=\frac12$ and it is true for any other value $D\in(0,1)$ if and only if the Riemann hypothesis is true, as derived by Lapidus and Maier in~\cite{LapMai91}.
}
\end{Remark}

\begin{Remark}
\emph{
({\it One sided bounds})
Under the hypothesis of Theorem~\ref{thm:one-dim}, also the following equivalence for the upper bounds holds regardless of any conditions on the lower bounds:
$$
(i)\quad \usM^D(F)<\infty \quad \Leftrightarrow \quad (ii)\quad \usS^D(F)<\infty \quad \Leftrightarrow \quad (iii) \quad \beta<\infty.
$$
Indeed, this is obvious from the proofs of Theorem~\ref{thm:one-dim}(a) and Proposition~\ref{lem:two-side}, for the equivalence of (i) and (iii) see also~\cite[Theorem~3.10]{LapPo93}. In the latter paper, it was also observed that the corresponding equivalence is not true for the lower bounds:
\begin{center}
$\alpha>0$ implies $\lsM^D(F)>0$ but not vice versa,
\end{center}
cf.~\cite[Theorem~3.11 and Example 3.13]{LapPo93}. A similar phenomenon was observed in \cite{w10} for the lower S-content:
\begin{center}
$\lsS^D(F)>0$ implies $\lsM^D(F)>0$ but not vice versa.
\end{center}
Fortunately, from the above proof of Theorem~\ref{thm:one-dim}(a), it is easily seen that at least the equivalence
\begin{align}
\lsS^D(F)>0 \quad \Leftrightarrow \quad \alpha>0
\end{align}
holds in general.
}
\end{Remark}

\begin{Remark}
\emph{
({\it Generalized contents of subsets of $\R$})
For the generalized contents a statement completely analogous to Theorem~\ref{thm:one-dim} can be derived by combining the results in Section~\ref{sec:general-gauge} with the results of He and Lapidus in \cite{LapHe1,LapHe2}, see in particular Theorems 2.4 and 2.6 in  \cite{LapHe1} (or Theorems 2.5 and 2.7 in \cite{LapHe2}), where corresponding equivalent assertions for the S-content can be added.
}
\end{Remark}

\end{document}